\documentclass[a4paper, 12pt]{article}
\usepackage{amssymb}
\usepackage{amssymb,amsbsy,float}
\usepackage{graphicx}
\usepackage{amsfonts}
\usepackage{latexsym}
\usepackage{bibtopic}

\newtheorem{theorem}{Theorem}

\newtheorem{proposition}[theorem]{Proposition}

\newenvironment{proof}[1][Proof]{\textbf{#1.} }{\ \rule{0.5em}{0.5em} \vspace{0.1 cm}}
\newcommand{\bx}{\mathbf{x}}

\newcommand{\by}{\mathbf{y}}%
\newcommand{\bp}{\mathbf{p}}%
\newcommand{\bq}{\mathbf{q}}
\newcommand{\bz}{\mathbf{z}}
\newcommand{\bh}{\mathbf{h}}
\newcommand{\bA}{\mathbf{A}}
\newcommand{\R}{\mathbb{R}}%
\newcommand{\N}{\mathbb{N}}%

\begin{document}
\title{Deterministic monotone dynamics and dominated strategies}
\author{Yannick Viossat}
\maketitle

\begin{abstract}
We survey and unify results on elimination of dominated strategies
by monotonic dynamics and prove some new results that may be seen as dual to those of Hofbauer and Weibull (J. Econ. Theory, 1996, 558-573)
on convex monotonic dynamics.\\

\end{abstract}
%
%

\section{Introduction}
A key issue in evolutionary game theory is whether evolutionary consideration lend support to rationality based concepts.
The most basic of these concepts is perhaps that 
strictly dominated strategies should not be played. Accordingly, whether evolutionary dynamics wipe out dominated
strategies has been studied by a number of authors (Akin, 1980; Nachbar, 1990; 
Dekel and Scotchmer, 1992; Samuelson and Zhang, 1992; Cabrales and Sobel, 1992; Hofbauer and Weibull, 1996; Cabrales, 2000; Berger and Hofbauer,
2005; Mertikopoulos and Moustakas, 2010; Hofbauer and Sandholm, 2011, and others). 
Many results concern deterministic monotonic dynamics and are proved in a very similar way. We unify these results, and prove
some new results on concave monotonic dynamics (defined in the
next section). These may be seen as a dual of Hofbauer and Weibull's (1996) results on convex monotonic dynamics.

The material is organized as follows. In the next section, we precise the framework and define several classes of
dynamics. The main results are given in section \ref{sec:res} for continuous-time dynamics and in section \ref{sec:discr} .
for discrete-time dynamics. Section \ref{sec:HW} gives examples of survival of strictly dominated strategies for single-population dynamics.  Section \ref{sec:nonmon} reviews more general dynamics and section \ref{sec:disc} discusses the concept of mixed strategy elimination.  Appendix \ref{sec:itel} deals with elimination of iteratively strictly dominated strategies. Finally, Appendices \ref{app:elim} and \ref{app:disc} provide details on some proofs. 

\section{Framework and classes of dynamics}
\label{sec:framework}

The setting we introduce is rather abstract in order to encompass
a variety of situations, including $1$-player games (decision problems), %
single population dynamics, and multipopulation dynamics where the different populations need not evolve according to the same dynamics.

We consider a large population of players from which individuals
are repeatedly and randomly drawn to play a game against some
unspecified opponent. For instance, if the underlying game is a $n$-person game, then we
focus on player $1$, which we call the focal player,  and consider players $2$ to $n$ as a single
entity, which we call the opponent (whether players $2$ to $n$ can
correlate their actions or not will not be relevant).

We assume that the set of pure strategies of
the focal player is finite and denote it by $I:=\{1,...,N\}$. The set $S_{opp}$ of strategies of the
opponent is compact. We denote by $S_N$ or $\Delta(I)$ the $N-1$ dimensional
simplex over $I$:
$$S_N=\Delta(I)=\left\{\bx \in \R_{+}^{N}, \sum_{i \in I} x_i=1\right\}$$
and by $\mbox{int}\,S_N:=\{\bx \in S_N:\forall i \in I, x_i
>0\}$ its interior. %
Let $x_i(t)$ denote the frequency of strategy $i$ at time $t$ in the focal population, and $\bx(t):=(x_i(t))_{i \in I}$ the vector of these frequencies.
Let $\by(t)$ be the opponent's strategy at time $t$.\footnote{If the game models an interaction with an individual of another population,
then $\by(t)$ is the mean strategy in this other population; if it models a symmetric interaction with an individual of the same population, then $\by(t)=\bx(t)$.}

The payoff of the pure strategy $i$ against strategy $\by$ is denoted $U_i(\by)$, where $U_i: S_{opp} \to \R$ is continuous.  The payoff of the mixed
strategy $\bp$ is $U_{\bp}(\by):=\sum_{i \in I} p_i U_{i}(\by)$. Vectors are in bold characters.\\

\textbf{Dynamics} We assume that the focal population adapts to the opponent's strategy through some biological or sociological process (imitation, selection,...). 
This is modeled by assuming that $\bx(t)$ follows a differential equation.
We focus on dynamics of the form:
\begin{equation}%
\label{eq:gendyn}%
\dot{x}_i(t)=x_i(t) \left[g_i(\bx(t),\by(t)) -\sum_{k \in I}
x_k(t) g_k(\bx(t),\by(t))\right]
\end{equation}
To ensure uniqueness of the solution through each initial condition, we assume that  the functions $g_i$ and $\by$ are locally Lipschitz. 
The simplex $S_N$ and its faces are then invariant under (\ref{eq:gendyn}).
%

This class of dynamics includes the payoff functional dynamics 
\begin{equation}
\label{eq:2popPF} %
\dot{x}_{i}=x_{i}\left[f(U_i(\by))-\bar{f}\right]
\end{equation}%
where $f: \R \to \R$ is Lipschitz continuous and $\bar{f}=\sum_{i \in I} x_i f(U_i(\by))$, with time indices suppressed.
An interpretation of (\ref{eq:2popPF}) is that $U_i(\by)$ is a material gain 
(food, territory, money,...) and $f(U_i(\by))$ its translation in terms of fitness. The replicator dynamics (Taylor and Jonker, 1978) corresponds to $f$ linear, that is when the payoffs
are directly measured in terms of fitness.\\
%

\textbf{Domination and elimination} As usual, we say that the mixed strategy $\bp$ {\em strictly dominates} the mixed strategy $\bq$ if $U_{\bp}(\by) > U_{\bq}(\by)$ for all $\by$ in $S_{opp}$. Throughout, we focus on strict domination, and write ``dominated" for ``strictly dominated". The pure strategy $i$ is {\em eliminated} under a solution $\bx(\cdot)$ of (\ref{eq:gendyn}) if $x_i(t) \to 0$ as $t \to +\infty$; it {\em survives} otherwise.
Following Samuelson and Zhang (1992) and Cabrales (2000), we say that the mixed strategy $\bq \in S_N$ is {\em eliminated} if $\min_{\{i \in I:
q_i>0\}} x_i(t) \to 0$ (or equivalently $\prod_{i \in I} x_i^{q_i}(t) \to 0$). We discuss the usefulness of this concept in section \ref{sec:disc}.\\

\textbf{Classes of dynamics} We want to find conditions on dynamics of type (\ref{eq:gendyn}) that ensure that dominated strategies are eliminated, or conversely, 
that dominated strategies survive in some games. To formulate these conditions, let us call $g_i(\bx(t),\by(t))$  the (unnormalized) \emph{growth rate} of strategy $i$ at time $t$. By analogy, we call {\em growth rate of the mixed strategy $\bp$} the quantity
%
\begin{equation} \label{eq:def-mixed-gr-rate} g_{\bp}(\bx,\by):=\sum_{i \in I} p_i g_i(\bx,\by)
\end{equation}
%
{\em Definition.} A dynamics (\ref{eq:gendyn}) is {\em aggregate monotonic} (Samuelson and
Zhang, 1992) if the growth rates of mixed strategies are ordered by their payoffs: for any mixed strategies $\bp$ and $\bq$ in $S_N$,
\begin{equation}
\label{eq:defAggMon}
\forall (\bx, \by) \in S_N \times S_{opp}, U_{\bp}(\by) > U_{\bq}(\by) \Rightarrow g_{\bp}(\bx, \by) >
g_{\bq}(\bx,\by) 
\end{equation}
It is \emph{monotonic} if (\ref{eq:defAggMon}) holds whenever $\bp$ and $\bq$ are pure strategies, and \emph{convex monotonic} (Hofbauer and Weibull, 1996) if (\ref{eq:defAggMon}) holds whenever $\bq$ is pure. It is \emph{concave monotonic} if (\ref{eq:defAggMon}) holds whenever $\bp$ is pure. Examples of convex and concave monotonic dynamics are discussed by Hofbauer and Weibull (1996).\footnote{Monotonicity (Samuelson and Zhang, 1992) is called {\em relative monotonicity} by Nachbar (1990), {\em order-compatibility 
} by Friedman (1991), 
and {\em payoff monotonicity} by Hofbauer and Weibull (1996). The definition of concave monotonic dynamics is new but 
related dynamics are studied by Bj\"ornerstedt (1995).}
%

To understand these names, note that a payoff functional dynamics (\ref{eq:2popPF}) with $f$ increasing is convex (resp. concave) monotonic
if and only if $f$ is convex (resp. concave). It is aggregate monotonic if and only if $f$ is linear, in which case (\ref{eq:2popPF}) is the replicator dynamics.\footnote{Samuelson and Zhang (1992
) have shown that in bimatrix games, for any
aggregate monotonic dynamics, there exists a positive speed
function $\lambda$ such that
$\dot{\bx}(t)=\lambda(\bx(t),\by(t))\dot{\bx}_{REP}(t)$, where
$\dot{\bx}_{REP}=x_i(U_{i}(\by)-U_{\bx}(\by))$. For single-population
dynamics ($\by(t)=\bx(t)$ $\forall t$), this implies that any
aggregate monotonic dynamics has the same orbits than the
replicator dynamics; for multi-population dynamics, this need not
be so, because the speed function is population specific.} %
\section{Intuition and main results}\label{sec:res}

%

\textbf{Intuition.} Convex monotonic dynamics favour mixed
strategies over pure ones in that a payoff advantage of a mixed strategy over a pure one always translates in a higher growth rate; thus they should eliminate dominated pure strategies, but not necessarily dominated mixed strategies. 
Similarly, concave monotonic dynamics favour pure strategies, 
thus should eliminate mixed strategies dominated by pure ones. 

Aggregate monotonic dynamics are a limit case: a better payoff always translates in a higher growth rate, hence all dominated strategies are eliminated. The same holds for monotonic dynamics when restricting attention to pure strategies. Formally:
\begin{proposition}
\label{prop:elim} Let $\bp$ and $\bq$ be two mixed
strategies. If $\bp$ strictly dominates $\bq$ then, under any aggregate monotonic dynamics and for any interior initial condition,  
$\bq$ is eliminated.

The same result holds for monotonic dynamics when $\bp$ and $\bq$ are both pure, for convex monotonic dynamics when $\bq$ is pure, and for concave monotonic dynamics when $\bp$ is pure.\footnote{For multipopulation dynamics, these results are due to Akin (1980) for the replicator dynamics, Nachbar (1990) for monotonic dynamics,  Samuelson and Zhang (1992) for aggregate monotonic dynamics and Hofbauer and Weibull (1996) for convex monotonic dynamics. The result on concave monotonic dynamics seems to be new.}
\end{proposition}
For generalizations to elimination of iteratively dominated strategies, see Appendix \ref{sec:itel}. 
Before providing a formal proof, we note that these results are sharp 
in the class of payoff functional dynamics (\ref{eq:2popPF}): 

\begin{proposition}
\label{prop:surv} Consider a payoff functional dynamics (\ref{eq:2popPF}). If $f$ is not convex,  then (\ref{eq:2popPF}) need not eliminate strictly dominated pure strategies: there exists a game with a strictly dominated pure strategy that survives for some interior initial conditions and some opponent's behaviors. 

Similarly, if $f$ is not concave, then (\ref{eq:2popPF}) need not eliminate mixed strategies dominated by a pure strategy;
if $f$ is not linear (resp. not increasing), then (\ref{eq:2popPF}) need not eliminate mixed strategies dominated by another mixed strategy; (resp. pure by pure).
\end{proposition}
The result on convex monotonic dynamics is due to Hofbauer and Weibull (1996). The others seem to be new.\\

\noindent \begin{proof}[Proof of proposition \ref{prop:elim}]
First consider aggregate monotonic dynamics:  let $\bp$ and $\bq$ be mixed strategies such that $\bp$ strictly dominates $\bq$. Since the dynamics is aggregate monotonic, it follows that $g_{\bp}(\bx, \by) > g_{\bq}(\bx, \by)$ for all $(\bx, \by) \in S_N \times S_{opp}$. By compactness of  $S_N \times S_{opp}$, there exists $\varepsilon>0$ such that
\begin{equation} \label{eq:proofagg}
g_{\bp}(\bx, \by)  - g_{\bq}(\bx, \by) \geq \varepsilon \quad \forall (\bx, \by) \in S_N \times S_{opp}
\end{equation}
Fix an interior solution $\bx(\cdot)$ and let $w(t):=\sum_{i \in I} (p_i -q_i) \ln x_i(t)$. By (\ref{eq:proofagg}), at any point in time, %
\begin{equation} \label{eq:dotw}
\dot{w}=\sum_{i \in I} (p_i - q_i) \frac{\dot{x}_i}{x_i}=\sum_{i \in I} (p_i -q_i)g_i(\bx,\by)=g_{\bp}(\bx,\by)-g_{\bq}(\bx,\by) \geq \varepsilon
\end{equation}
Therefore, $\lim_
{t \to + \infty} w(t)=+ \infty$. Since $\sum_{i \in I} p_i \ln x_i(t)  \leq 0$ (as $x_i \leq 1$), this implies that  $\sum_{i \in I} q_i \ln x_i(t) \to -\infty$. Therefore $\prod_{i \in I} x_i(t)^{q_i} \to 0$ and $\bq$ is eliminated.

The proof for the other classes of dynamics is the same, up to replacement of $\bp$, $\bq$, or both, by pure strategies.
\end{proof}

\noindent \begin{proof}[Proof of proposition \ref{prop:surv}]
The case $f$ nonincreasing is obvious. The case $f$ nonlinear follows from the cases $f$ nonconvex and $f$ nonconcave, which we now deal with. 
%
%
Assume first that $f$ is not convex. Since $f$ is continuous (as assumed throughout), there exist real numbers $a$, $b$, and $\varepsilon>0$ such that
\begin{equation}
\label{eq:nonconcave}
\frac{f(a) + f(b)}{2} < f\left(\frac{a+b}{2} - \varepsilon \right) %
\end{equation}%
Consider the $3 \times 2$ game with payoff matrix:
\begin{equation}
\label{game:surv}
\begin{array}{cc}
  & \begin{array}{cc}
    L \hspace{0.5 cm} & \hspace{0.5 cm} R \\
    \end{array} \\
\begin{array}{c}
 T \\
 M \\
 B \\
\end{array}
& \left(\begin{array}{cc}
           a            &           b            \\
 \frac{a+b}{2} -\varepsilon & \frac{a+b}{2}-\varepsilon \\
           b            &           a            \\
\end{array}\right)\\
\end{array}
\end{equation}
The pure strategy $M$ is strictly dominated by $\bq=(1/2, 0, 1/2)$. But if, approximately, $\by(t)=L$
half of the time and $\by(t)=R$ the other half, then the average (unnormalized) growth rate of strategies $T$ and $B$
is close to $[f(a)+f(b)]/2$, while the average growth rate of strategy $M$ is
$f([a+b]/2-\varepsilon) > [f(a)+f(b)]/2$.
It follows that $x_M \to 1$, hence strategy $M$ survives.

If $f$ is not concave, then take $a$, $b$, and $\varepsilon>0$ such that $\frac{f(a) + f(b)}{2} > f\left(\frac{a+b}{2} + \varepsilon \right)$.
Consider the same game as  (\ref{game:surv}), but with $-\varepsilon$ changed into $+\varepsilon$. Strategy $\bq$ is strictly dominated by $M$.
But against the same opponent's behaviour, the average growth rate of strategy $\bq$ is now higher than the average growth rate of $M$.
It follows that $x_M \to 0$. Choosing $y(t)$ periodic and respecting the symmetry
between strategies $T$ and $B$ ensures that $x_T/x_B$ remains bounded. Together with $x_M \to 0$, this implies that $\liminf_{t \to + \infty} x_T (t) x_B (t) >0$,
hence $\bq$ is not eliminated. For details see Appendix \ref{app:elim}.
\end{proof}

As detailed in Appendix C, a variant of this proof shows that proposition \ref{prop:surv} extends to the wider class of payoff functional dynamics considered by Hofbauer and Weibull (1996).

\section{Discrete-time dynamics}
\label{sec:discr} %
Dekel and Scotchmer (1992) show that the discrete-time replicator dynamics need not eliminate pure strategies
strictly dominated by mixed strategies. Thus, proposition \ref{prop:elim} does not extend to
discrete-time dynamics. To see why, consider the
discrete-time dynamics:
\begin{equation}
\label{eq:discrete-gen} x_i(n+1)=x_i(n) \frac{C+
g_i(\bx,\by)}{C + \sum_{k} x_k g_k(\bx,\by)}
\end{equation}
where $\bx=\bx(n)$, $\by=\by(n)$, and the constant $C > - \min_{i,\mathbf{x},\mathbf{y}} g_i(\bx, \by)$ may be interpreted as background fitness (Maynard-Smith, 1982).
Eq. (\ref{eq:discrete-gen}) is equivalent to
\begin{equation}
\label{eq:discrete-gen-bis} x_i(n+1) - x_i(n)=x_i(n)
\frac{g_i(\bx,\by)- \sum_{k} x_k g_k(\bx,\by)}{C + \sum_{k} x_k
g_k(\bx,\by)}
\end{equation}
which is the Euler discretization of (\ref{eq:gendyn}) with step size $1/[C + \sum_{k} x_k g_k(\bx,\by)]$. The discrete-time replicator dynamics
(Maynard-Smith, 1982) corresponds to $g_i(\bx,\by)=U_{i}(\by)$. Dekel and Scotchmer (1992) take $C=0$.

Proposition \ref{prop:elim} and \ref{prop:surv} are based on the evolution of the quantity $w:=\sum_{i \in I} (q_i -p_i) \ln
x_i$, where $\bp$ and $\bq$ are two mixed strategies. In continuous-time, we had (\ref{eq:dotw}): 
$\dot{w}=\sum_{i \in I} (q_i - p_i)g_i(\bx,\by)$. But it follows from (\ref{eq:discrete-gen}) that in discrete-time, 
\begin{equation}
\label{eq:discrete-dotw} w(n+1) - w(n)=\sum_{i \in I} (q_i -
p_i) \tilde{g}_i(\bx,\by) \mbox{ with } \tilde{g}_i=\ln\left(C + g_i\right).
\end{equation}
It follows that propositions \ref{prop:elim} and \ref{prop:surv} hold for (\ref{eq:discrete-gen}) if we replace $g_i$ by $\tilde{g}_i$ when defining our classes of dynamics. This was first understood, and shown for aggregate monotonic dynamics, by Cabrales and Sobel (1992).

For payoff functional dynamics 
\begin{equation}
\label{eq:discPF}
x_i(n+1)=x_i(n) \frac{C+ f(U_{i}(\by))}{C + \sum_{k} x_k f(U(k,\by))}
\end{equation}
we get:
\begin{proposition}
\label{prop:discrete-conv-neg}
The dynamics (\ref{eq:discPF}) eliminates pure (mixed) strategies strictly dominated by 
mixed (pure) strategies if and only if $\ln(C + f)$ is increasing and convex (concave). 
It eliminates mixed strategies strictly dominated by mixed strategies if and only if $\ln(C + f)$ 
is increasing and linear.  
\end{proposition}
Since the logarithm is concave, dynamics that are aggregate monotonic in continuous-time become concave monotonic in discrete-time, and thus do not eliminate dominated pure strategies. Dekel and Scotchmer's (1922) example is an instance of this general result.

For large $C$ (a fine discretization), $\tilde{g}_i \simeq \ln C + g_i/C$ and discretizing concavifies less the dynamics.
If we first fix the game and the functions $g_i$ and then take $C$ large enough, or if we let $C$ depend on $n$, with $C_n \to +\infty$ (and $\sum_{n} 1/C_n = + \infty$ for the dynamics not to stop),
then proposition \ref{prop:elim} holds also in the discrete case, 
without modifying the definitions of the classes of dynamics (see Appendix \ref{app:disc} and Cabrales and Sobel, 1992). For more on discrete time dynamics, see Cabrales and Sobel (1992) and Bj\"ornerstedt et al. (1996).\footnote{ Bj\"ornerstedt et al. (1996) consider an overlapping-generations dynamics which has as limit cases the discrete-time and the continuous-time replicator dynamics, for respectively no overlap and full overlap. The degree of overlap plays a role similar to that of the background fitness $C$ in (\ref{eq:discrete-gen} ). They show that, for a fixed game, if the degree of overlap is large enough, strictly dominated strategies are eliminated.}

\section{Survival of strictly dominated strategies in single-population dynamics}
\label{sec:HW}

Proposition \ref{prop:surv} is easy to prove because the opponent's behaviour may be chosen ad-hoc.
This section shows that similar results may be obtained with more realistic opponent's behaviours.
We first recall a result of Hofbauer and Weibull (1996).

\subsection{Hofbauer and Weibull's result}
\label{subsec:HW}
Consider a single population dynamics with a payoff functional form (\ref{eq:2popPF}), with $f$ $C^1$.
Up to a change of time, and suppressing time indices, it may be written as  $\dot{x}_i=x_i [f(U_i) - \bar{f}]$
with $U_i=U_{\bx}(\bx)$ and $\bar{f}= \sum_i x_i f(U_i)$. Hofbauer and Weibull (1996) show that if $f$ is not convex,
then there are games where pure strategies strictly dominated by mixed strategies survive, for many interior initial conditions,  The game they consider is a $4 \times 4$ two-player symmetric game.
It is built by adding a fourth strategy to a generalized Rock-Paper-Scissors (RPS) game.
The payoff matrix has the following structure:
\[
\left(
\begin{array}{cccc}
a  &  c & b  &  \gamma \\
b  & a  & c  & \gamma \\
c  & b  & a & \gamma \\
a + \beta & a + \beta & a + \beta & 0
\end{array}
\right)
\]
where: (i) $c < a < b$, (ii) $(a + b + c)/3 > a + \beta$ and $\gamma > 0$ , and (iii) $\beta > 0$.

By (i), the $3 \times 3$ base game is a RPS game, and by (ii), strategy $4$ is strictly dominated by $\bp=(1/3, 1/3, 1/3, 0)$.
Since $\beta$ and $\gamma$ are positive, there is a rest point $\bz_i$ in the relative interior of the edge connecting strategy
$i \in \{1,2,3\}$ to strategy 4. It attracts any solution starting in the relative interior of this edge,
and actually, any solution starting in the relative interior of the face $x_{i+1}=0$, where $i+1$ is counted modulo $3$. It follows that the $\bz_i$ form a heteroclinic cycle, that is, a loop of saddle rest points and orbits connecting them. 
Conditions (ii) and (iii) imply that $(a + b + c)/3 > a$ (or equivalently, $a < (b+c)/2$). This implies that, under the replicator dynamics,
the base game in an inward cycling RPS game (Hofbauer and Sigmund, 1998, p93; Gaunersdorfer and Hofbauer, 1995, p291). It may be shown that under the replicator dynamics
(i.e. $f$ strictly increasing and linear), any solution starting in the interior of the state space converges to the Nash equilibrium $\bp=(1/3, 1/3, 1/3, 0)$. In particular, strategy $4$ is eliminated. However, if the function $f$ is strictly increasing
but not convex, then one may find such payoffs that furthermore satisfy $f(a) > [f(b) + f(c)]/2$.
This means that in the RPS base game, even though the replicator dynamics is inward cycling,
the current dynamics is outward cycling close to the RPS cycle; thus, this cycle is asymptotically stable
on the face $x_4=0$. As detailed by Hofbauer and Weibull, this implies that for $\beta$ small enough,
the heteroclinic cycle between the $\bz_i$ is asymptotically stable. But along this cycle, $x_4>0$. Thus, 
though strictly dominated, strategy 4 survives for an open set of initial conditions.

\subsection{A dual result.}
\label{sec:dr}

Our aim is to give a dual example, showing that if $f$ is strictly increasing and sufficiently regular but not concave,
then mixed strategies that are strictly dominated by a pure strategy need not be eliminated.

To simplify the exposition, we assume that $f$ is differentiable, with a strictly positive derivative, and that $f$ is strictly convex. The strict convexity assumption may be replaced by any assumption that guarantees that $f$ is strictly convex over some range, as we can always locate all the payoffs in this range. For instance, $f$ $C^2$ and not concave suffices.

Consider the game with payoff matrix
 \begin{equation}
 \label{game:dual}
\bA = \left(
\begin{array}{cccc}
a  &  c & b  & m- \gamma \\
b  & a  & c  & m-\gamma \\
c  & b  & a & m-\gamma \\
m + \beta & m + \beta & m + \beta & m
\end{array}
\right)
\end{equation}
where $m=(a+b+c)/3$. Assume that : (i) $c < a < b$, (ii) $\beta > 0$ and $\gamma > 0$, and (iii)  $a > (b+c)/2$.
Condition (i) means that the $3 \times 3$ base game is a RPS game; condition (ii) implies that strategy $4$ strictly dominates the strategy $\bp=(1/3, 1/3, 1/3, 0)$;
condition (iii) implies that in the RPS base game, the replicator dynamics cycles outwards.

Because $f$ is strictly convex, we may find such payoffs that furthermore satisfy: (iv) $f(a) < [f(b) + f(c)]/2$.
Consider first the dynamics on the face $x_4=0$, that is, in the RPS base game. Condition (iv) implies that on this face,
even though the replicator dynamics is outward cycling, our dynamics is inward cycling close to the RPS heteroclinic cycle,
i.e. the relative boundary of this face is repulsive (Hofbauer and Sigmund, 1998; Gaunersdorfer and Hofbauer, 1995). Using a Taylor expansion of $f$ near $m$, it may be shown that, close to the rest point
$\bp=(1/3, 1/3, 1/3, 0)$, the dynamics behaves as if $f$ was linear, that is, as the replicator dynamics. Due to condition (3) this implies that close to $\bp$
(except exactly at $\bp$), the dynamics is outward cycling: $x_1x_2x_3$
strictly decreases along trajectories (Figure 1).
\footnote{Technically, fix a solution $\bx(\cdot)$ in the relative interior of the face $x_4=0$.
Let $v(t)=(x_1x_2x_3)^{1/3}$ and $w(t)=\ln v(t)$. Let $\bh(t)=\bx(t)-\bp$. Then $\dot{w}(t)=- [a-(b+c)/2] || \bh(t) ||^2 + o(|| \bh(t) ||^2)$,
hence near $\bp$, by condition (iii), $x_1x_2x_3$ decreases along trajectories.
To obtain this expression of $\dot{w}$, note that $(\bA\bx)_i - m=(\bA\bh)_i$, thus $f((\bA\bx)_i)=m + f'(m) (\bA\bh)_i  + o([\bA\bh)_i])$.
Using this and $\sum_{i} h_i=0$, we get $\dot{w}=  - f'(m) \bh \cdot \bA\bh + o( || \bh ||^2)$.
A standard computation shows that $\bh \cdot \bA\bh=[a-(b+c)/2] || \bh(t) ||^2$ hence the result.
This may be seen as a variant of exercise 8.1.1 in Hofbauer and Sigmund (1998).} 

By compactness, this implies that there exists $\rho \in ]0,1/27[$ and a positive
constant $C$ such that, if $x_4=0$ and $x_1x_2x_3=\rho$, then $d(x_1x_2x_3)/dt < -C$. By continuity of $f$,
this implies the existence of a small positive $\varepsilon$ such that, if $x_4 \leq \varepsilon$
and $x_1x_2x_3=\rho$, then $d(x_1x_2x_3)/dt < -C/2$. Furthermore, for $\beta=0$, except if $x_1=x_2=x_3$,
the payoff of strategy $4$ is less than the mean payoff: $(\bA\bx)_4 < \bx \cdot \bA\bx$ (Viossat, 2007), hence $\dot{x}_4<0$ by convexity of $f$.
By a compactness and continuity argument, this implies that for $\beta$ small enough,
if $\bx$ belongs to the compact set  $K=\{\bx \in S_4 : x_1x_2x_3 \leq \rho \mbox{ and } x_4  \leq \varepsilon$\},
then $\dot{x}_4< -C'$ for some positive constant $C'$.

Now consider any solution with initial condition in $K$.  The solution cannot leave $K$, which implies that it converges to the face $x_4=0$.
Moreover, since on this face the relative boundary is repulsive, this implies that $\liminf x_1x_2x_3 >0$. Thus, though $\bp=(1/3, 1/3, 1/3,0)$
is strictly dominated by strategy $4$, it is not eliminated.\footnote{This proof is similar to the proof that the Brown-von-Neumann-Nash dynamics
may eliminate all strategies used in correlated equilibria (Viossat, 2008)}

\begin{center}
\includegraphics[scale=0.5]{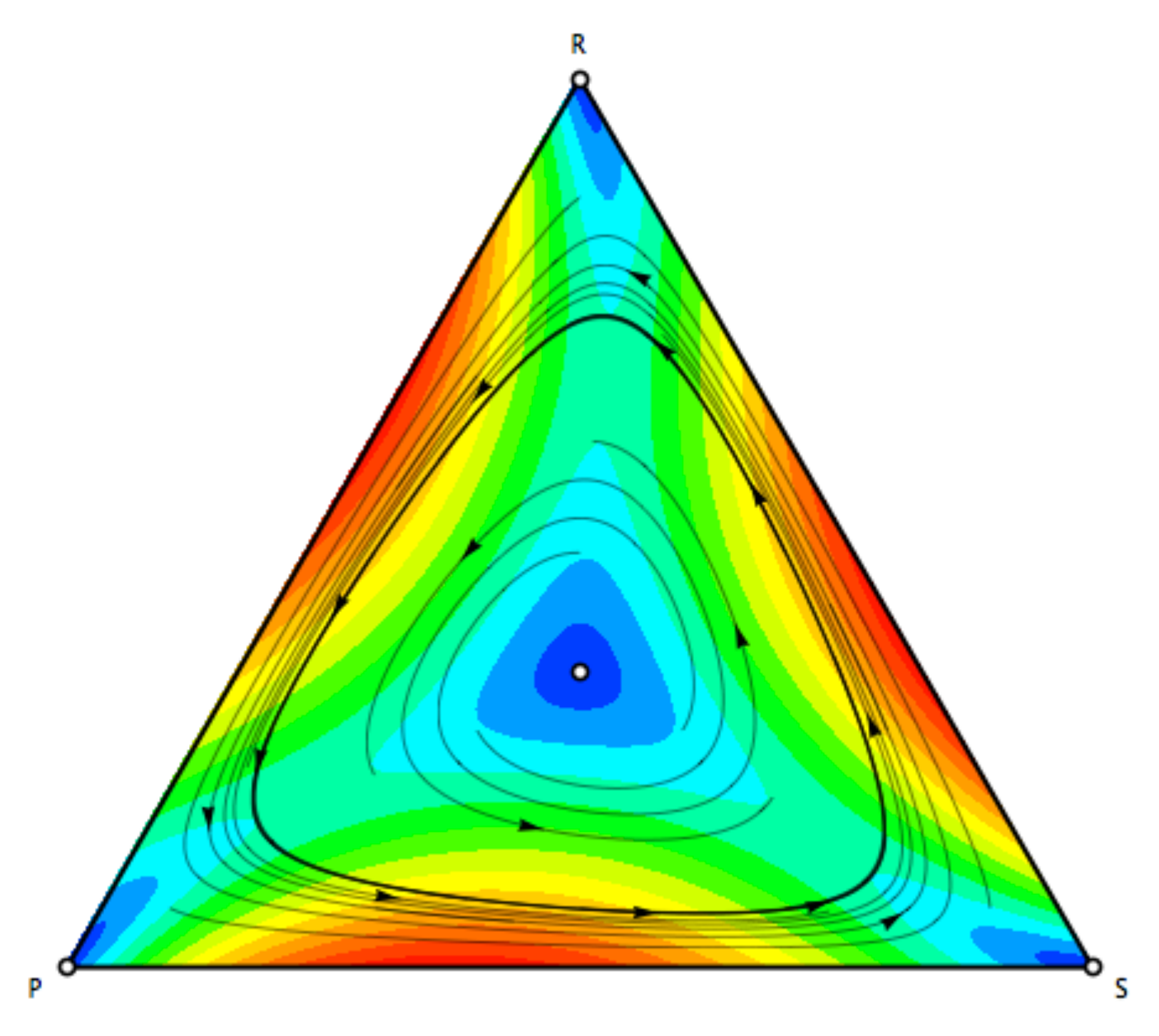}
\end{center}
\begin{small}
Figure 1: Dynamics on the face $x_4=0$. Both the center and the boundary are repulsive.  In the neighborhood of this face, for suitable payoffs, $x_4$ decreases outside a 
small neighborhood of $(1/3, 1/3, 1/3, 0)$. Thus the limit cycle (or annulus) in the interior of this face is asymptotically stable in the whole 
game. Graphic produced with Dynamo (Sandholm et al., 2011).
\end{small}

\subsection{Discrete-time analog}
For discrete-time dynamics (\ref{eq:discPF}), the survival examples of sections \ref{subsec:HW} and \ref{sec:dr} go through if the assumptions on the convexity of $f$ are replaced by the corresponding assumptions on $\ln f$. 
Indeed, for reasons similar to those of section \ref{sec:discr}, the stability condition of the RPS cycle of the base game for  (\ref{eq:discPF}) is : $\ln f(a)  > [\ln f(b) + \ln f(c)]/2$ (Gaunersdorfer, 1992, section 5 ; Hofbauer and Schlag, 2000, p533 ; see also Weissing, 1991). For Hofbauer and Weibull's result, the proof goes through with no other changes. For our dual result, two other ingredients are needed: first, $(\bA\bx)_4 < \bx \cdot \bA\bx$ should imply  $f[(\bA\bx)_4) < \sum_i x_i f[(\bA\bx)_i]$; but if $\ln f$ is convex, so is $f$, hence this holds. Second, in the RPS base game, $x_1x_2x_3$ should decrease near the equilibrium. In continuous time, when $a > [b+c]/2$, this holds for any $C^1$ monotonic dynamics; but the discrete-time dynamics trajectories are tangent to their continous-time counterpart and the set of mixed strategies $\bx$ such that $x_1x_2x_3 \geq \lambda$ is convex; thus, under the discrete-time dynamics, $x_1x_2x_3$ decreases a fortiori (see also Hofbauer and Schlag, 2000, p534). The rest of the proof is unchanged.

\section{Other dynamics}
\label{sec:nonmon}
\indent \emph{Continuous set of pure strategies} Elimination of dominated strategies by monotonic dynamics is generalized to games with a continuous set of pure strategies by Heifetz et al (2007a, 2007b) and by Cressman and Hofbauer (2005) and Cressman et al (2006) [see also, Cressman, 2005] for the replicator dynamics. 

\emph{Stochastic dynamics} Cabrales (2000) studies a stochastic version of the replicator dynamics and gives conditions under which dominated strategies are eliminated. Mertikopoulos and Moustakas (2010) show that for another version of the stochastic replicator dynamics, dominated strategies are always eliminated.

\emph{Best-reply dynamics. } Under the best-reply dynamics (Gilboa and Matsui, 1991; Matsui 1992),
the frequency of pure strategies that are not best replies to the current strategy of the opponent decreases exponentially.
It follows that strictly dominated pure strategies are eliminated. This extends to iteratively strictly dominated pure strategies.
However, strictly dominated mixed strategies need not be eliminated. This follows from (Viossat, 2008) or from a variant of the proof of proposition (\ref{prop:surv}). These results are not surprising, as 
the best reply dynamics may be seen as a limit case of convex monotonic dynamics (Hofbauer and Weibull, 1996)

\emph{Survival of pure strategies strictly dominated by other pure strategies.}
Building on Berger and Hofbauer (2006), Hofbauer and Sandholm (forthcoming)  [see also Sandholm (2011)] show that under a large class of dynamics, a pure strategy strictly dominated by another pure strategy may survive. Their key feature is that, except at Nash equilibria, the growth rate of an unused best-reply is always positive. Combined with regularity requirements,
this gives a small advantage to rare strategies which allows to maintain strictly dominated strategies at nonvanishing frequencies.\footnote{An example of a game and a 
dynamics such that a pure strategy dominated by {\em all} other pure strategies survives for most initial conditions is given by  Bj\"ornerstedt et al. (1996).}

\section{Discussion} \label{sec:disc}

One of the contribution of this article is to study elimination of dominated mixed strategies by concave monotonic dynamics.  Elimination of dominated mixed strategies has a natural interpretation if we assume that agents may play mixed strategies,
but even if we think that agents play only pure strategies, we find the concept useful on at least three grounds. 
First, to compare the outcome of evolutionary dynamics with rationalizable strategies (Bernheim, 1984; Pearce, 1984). 
Second, to clarify the logic of Hofbauer and Weibull's (1996) results by showing that elimination of pure strategies by 
convex monotonic dynamics and of mixed strategies by concave monotonic dynamics are two faces of the same coin. 
Third, to analyse the behaviour of some dynamics.

Consider for instance a single-population dynamics modeling the evolution of behavior in the two-player symmetric game
\[\left(\begin{array}{ccc}
3 & 0 & 0 \\
0 & 3 & 0 \\
2 & 2 & 1
\end{array}\right)\]
The mixed strategy $\bq=(1/2, 1/2, 0)$ is strictly dominated by the third pure strategy. Thus, if the dynamics is concave monotonic,
then for any interior initial condition $x_1(t)x_2(t) \to 0$, hence the solution converges to the union of the edges $x_1=0$ and $x_2=0$.
Together with the dynamics on these edges (which is qualitatively the same for all monotonic dynamics), this implies that any interior solution converges to one of the three pure Nash equilibria, or to one of the rest points in the relative interior of these edges, that is $(1/2, 0, 1/2)$ and  $(0, 1/2, 1/2)$. This need not be the only way to prove this result, but probably the most elementary.

\begin{appendix}

\section{Iterative elimination}
\label{sec:itel}
Proposition (\ref{prop:elim}) extends to elimination of iteratively strictly dominated strategies. For simplicity, we focus on two-player finite games. The extension to $n$-player games is immediate.

Let $I$ and $J$ denote the finite sets of pure strategies of player $1$ and $2$, respectively, and for any finite set $K$, let $\Delta(K)$ denote the simplex of probability distribution over $K$.
Thus, the set of mixed strategies of player $1$ is $\Delta(I)$.

Let $I^0=\tilde{I}^0=I$ and $S^0=\tilde{S}^0=\Delta(I)$. Similarly, let $J^0=\tilde{J}^0=J$ and $T^0=\tilde{T}^0=\Delta(J)$. Inductively, let $I^{k+1}$ (resp. $\tilde{I}^{k+1}$, $S^{k+1}$, $\tilde{S}^{k+1}$) denote the set of strategies
in $I^k$ (resp. $\tilde{I}^{k}$, $S^{k}$, $\tilde{S}^k$)
that are not strictly dominated by any strategy in $I^k$ (resp. $\Delta(\tilde{I}^{k})$, $I \cap S^{k}$, $\tilde{S}^k$)
when player 2 chooses strategies in
$J^k$ (resp. $\tilde{J}^k$, $T^k$, $\tilde{T}^k$). Similar definitions apply to player 2.\\

\noindent {\em Definition.} A pure strategy $i \in I$ (resp. a mixed strategy $\bx \in \Delta(I)$)
is \emph{iteratively strictly dominated by pure strategies}
if there exists $k$ in $\N$ such that $i \notin I^{k}$ (resp. $\bx \notin S^k$).
A pure strategy $i \in I$ (resp. a mixed strategy $\bx \in \Delta(I)$) is \emph{iteratively strictly dominated}
if there exists $k$ in $\N$ such that $i \notin \tilde{I}_k$ (resp. $\bx \notin \tilde{S}^{k}$). \footnote{ \label{ft:iter} Since a pure strategy is also a mixed strategy,
there seems to be two definitions of iteratively strictly dominated pure strategies,
but they are equivalent. This is because iterative elimination of dominated
mixed strategies boils down to iterative elimination of dominated pure strategies
followed by one round of elimination of dominated mixed strategies, as is easily shown.
Similarly, iterative elimination of mixed strategies dominated by pure strategies boils down to 
iterative elimination of pure strategies dominated by other pure strategies, followed by one round of elimination of mixed strategies dominated by pure ones. Thus, both definitions of pure strategies iteratively strictly dominated by pure strategies are also equivalent.}

\begin{proposition} For every interior initial condition: if both $\bx(t)$ and $\by(t)$ follow a monotonic (resp. concave monotonic) dynamics, then every pure (resp. mixed) strategy iteratively strictly
dominated by pure strategies is eliminated; if both $\bx(t)$ and $\by(t)$ follow a convex  (resp. aggregate) monotonic dynamics, then every iteratively strictly
dominated pure (resp. mixed) strategy is eliminated.
\end{proposition}

\begin{proof}
We first prove the result on convex monotonic dynamics. It follows from proposition \ref{prop:elim} that the pure strategies in $\tilde{I}^{0}\backslash \tilde{I}^{1}$
and in $\tilde{J}^{0}\backslash \tilde{J}^{1}$ are eliminated. By induction assume that for $k$ in $\N$, the strategies in
$\tilde{I}^{0}\backslash \tilde{I}^{k}$ and in $\tilde{J}^{0}\backslash \tilde{J}^{k}$ are eliminated, hence in particular
\begin{equation}
\label{eq:iterdom} \forall \eta>0, \exists T \in \R, \forall t
\geq T, \max_{j \in J \backslash \tilde{J}^k} y_j(t) \leq \eta
\end{equation}
Let $i \in \tilde{I}^{k}\backslash \tilde{I}^{k+1}$. Since strategy $i$ is
strictly dominated in the game restricted to $\tilde{I}^k \times \tilde{J}^k$,
there exists $\varepsilon>0$, $\eta>0$ and $\bp \in S_N$ such that
$$\max_{j \in J \backslash \tilde{J}^k} y_j \leq \eta \Rightarrow
U_{i}(\by(t)) < U_{\bp}(\by(t)) + \varepsilon$$
Therefore, it follows from (\ref{eq:iterdom}) and the proof of proposition \ref{prop:elim} that strategy $i$ is eliminated. The result follows.

The proof for monotonic dynamics is identical up to replacement of $\tilde{I}^k$, $\tilde{J}^k$ by $I^k$, $J^k$, and of ``$p \in S_N$" by ``$i \in I$". Concave monotonic dynamics eliminate pure strategies iteratively strictly dominated by pure strategies, as any monotonic dynamics, and then mixed strategies that become dominated by pure ones after this elimination, by an argument similar to the iteration step for convex monotonic dynamics. This and footnote \ref{ft:iter} proves the result. The proof for aggregate monotonic dynamics relies similarly on the fact that they are convex monotonic and footnote \ref{ft:iter}.  
\end{proof}

These results are due to Nachbar (1990) for monotonic dynamics,  Samuelson and Zhang (1992) for aggregate monotonic dynamics and Hofbauer and Weibull (1996) for convex monotonic dynamics. 

\section{Details on survival of dominated mixed strategies under non concave monotonic dynamics}
\label{app:elim}
\subsection{For payoff functional dynamics (\ref{eq:2popPF})}
\label{app:elim1}
If $f$ is not concave, then there exist reals $a$ and $b$ such
that
$$\frac{f(a) + f(b)}{2} >  f\left(\frac{a+b}{2}\right)$$

By continuity of $f$, there exists $\varepsilon>0$ such that:
\begin{equation} \label{eq:nonconcave2}
\alpha:= \frac{f(a) + f(b)}{2} - f\left(\frac{a+b}{2} + \varepsilon \right) >0%
\end{equation}%
Assume that player $1$ plays a $3 \times 2$ game with payoff
matrix:
$$\begin{array}{cc}
  & \begin{array}{cc}
    L \hspace{0.5 cm} & \hspace{0.5 cm} R \\
    \end{array} \\
\begin{array}{c}
 1 \\
 2 \\
 3 \\
\end{array}
& \left(\begin{array}{cc}
           a            &           b            \\
 \frac{a+b}{2} +\varepsilon & \frac{a+b}{2}+\varepsilon \\
           b            &           a            \\
\end{array}\right)\\
\end{array}
$$
Let $y_l(t)$ denote the probability at time $t$ that the opponent
chooses the left column. Fix some large positive constant $T$
(more precisely, $T>(2C+1)/\alpha + 1$ with $C=\max_{[a,b]} |f|$;
this condition will appear later). Assume that the function
$y_l(\cdot)$ is $2T$-periodic with $y_l(t)=1$ if $t \in [0,T-1]$,
$y_l(t)=0$ if $t \in [T, 2T-1]$, and linear variation on $[T-1,T]$
and $[2T-1,2T]$.

Assume that $\bx(0) \in \mbox{int} S_N$ and let
$$w(t):=\ln\left(\frac{x_2}{\sqrt{x_1x_3}}\right)(t)$$
Letting $g_i(t):=f(U_{i}(\by(t))$, we have:
$$\dot{w}(t)=g_2(t) - \frac{g_1(t)+g_3(t)}{2}$$

hence
$$w((k+1)T)-w(kT) =\int_{kT}^{(k+1)T} \left(g_2(t)-\frac{g_1(t)+
g_3(t)}{2}\right)\, dt$$

Since for all $t$, $g_2(t)=f((a+b)/2 +\varepsilon)$ and since
$$\forall t \in [kT,(k+1)T-1], \, g_1(t)+ g_3(t)=f(a) + f(b)$$

it follows from (\ref{eq:nonconcave2}) that:
$$w((k+1)T)-w(kT) \leq -(T-1)\alpha + 2C$$
with $C=\max_{[a,b]} |f|$. Since we assumed $T>(2C+1)/\alpha + 1$, it follows that
$$w((k+1)T) \leq w(kT) -1$$
Therefore $w(kT) \to - \infty$ as $k \to +\infty$. Since the variation of $w$ between $kT$ and $(k+1)T$ is
bounded (less than $2CT$), it follows that $w(t) \to - \infty$ hence $x_2(t) \to 0$ as $t \to +\infty$.
Furthermore, it is easy to see by the same kind of computation that $x_1(t)/x_3(t)$ is $2T$-periodic. It
follows that $x_1x_3 \nrightarrow 0$. Actually, as is easily seen,
$$\inf_{t \in \R_+} x_1(t)x_3(t)=\min_{t \in [0,2T]} x_1(t)x_3(t)
>0$$
Therefore, though $\bq=(1/2, 0, 1/2)$ is strictly dominated by strategy $2$, the mixed strategy $\bq$ is not
eliminated.\footnote{Note also that for $T$ sufficiently large or $x_1(0)$ sufficiently close to $x_3(0)$,
there exists $t$ in $[0,2T]$ such that $x_1(t)=x_3(t)$, hence $x_1(t +2kT)=x_3(t+2kT)$ for all $k \in \N$.
Together with $x_2(t) \to 0$, this implies that $\limsup_{t \to +\infty} x_1(t)x_3(t)=1/4$.}

\subsection{For the wider class of payoff functional dynamics considered by Hofbauer and Weibull (1996)}
\label{app:elim2}

In the context of $n$-player normal form games (hence with the opponent corresponding to players $2$ to $n$), Hofbauer and Weibull (1996) consider payoff functional dynamics of the form (time indices suppressed):
\begin{equation}
\label{eq:HWPF} %
\dot{x}_{i}=\lambda(\bx, \by) x_{i} \left[f(U_i(\by))-\bar{f}\right]
\end{equation}%
This is more general than (\ref{eq:2popPF}) because the speed of the dynamics may now depend on the population's and opponent's mean strategies. Assuming that $\lambda$ is continuous, the previous proof is easily adapted to deal with these more general dynamics. 
 
Fix a large constant $T$. Let $(t_n)$ be a strictly increasing sequences of times, with $t_0=0$. Assume that :

(i) $y(t)$ is piecewise linear, with for all $n \in \N$ : $y_L(t)=1$ on $[t_{6n}, t_{6n+1}]$, $y_L(t)=0$ on $[t_{6n+3}, t_{6n+4}]$, $y_L(t_{6n+2})=y_L(t_{6n+5})=1/2$, and linear variation 
on $[t_{6n+1}, t_{6n+2}]$, $[t_{6n+2}, t_{6n+3}]$, $[t_{6n+4}, t_{6n+5}]$ and $[t_{6n+5}, t_{6n+6}]$.
Furthermore :

(ii) Choose $t_{6n+1}$ and $t_{6n+4}$ such that  
$$\int_{t_{6n}}^{t_{6n+1}} \lambda(x(s),y(s)) ds = \int_{t_{6n+3}}^{t_{6n+4}} \lambda(x(s),y(s)) ds=T ; $$ 

(iii) Let $t_{6n+3} = t_{6n+1}+1$ and $t_{6n+6} = t_{6n+4} + 1$;
 
 (iv) if $a=b$, choose $t_{6n+2}$ and $t_{6n+5}$ arbitrarily in $]t_{6n+1}, t_{6n+3}[$ 
 and $]t_{6n+4}, t_{6n+6}[$, respectively.  Otherwise, assume w.l.o.g. $a < b$, 
 note that $g_1- g_3$ is negative on $]t_{6n+1}, t_{6n+2}[$ and positive on  $]t_{6n+2}, t_{6n+3}[$, 
 and choose $t_{6n+2}$ such that $\int_{t_{6n+1}}^{t_{6n+3}} [g_1(\bx(s), \by(s)) - g_3(\bx(s), \by(s))] ds =0$. 
 Similarly, choose $t_{6n+5}$ such that $$\int_{t_{6n+4}}^{t_{6n+6}} [g_1(\bx(s), \by(s)) - g_3(\bx(s), \by(s))] ds =0 .$$

The same kind of computation as in the case $\lambda \equiv 1$ show that 
$$w(t_{6n+6}-w(t_{6n}) < -2T \alpha +  4 v_{max} C$$ 
where $v_{max}= \max_{(\bx, \by) \in S_N \times S_{opp} } \lambda(\bx, \by)$, 
and $C=\max_{[a,b]} |f|$. Thus for $T > 2v_{max} C / \alpha$, $w(t) \to_{t \to +\infty} - \infty$. 
Thus $x_2(t) \to 0$ and $x_1(t) + x_3(t) \to 1$. Conditions (ii) and (iv) ensure that 
$[x_1/x_3] (t_{6n+6})= [x_1/ x_3](t_{6n})$. Together with  $x_1(t) + x_3(t) \to 1$, this implies that 
$\limsup x_1(t)x_3(t) \geq x_1(0)x_3(0) >0$, hence $\bq$ is not eliminated.   	

%
%

\section{More on discrete dynamics}
\label{app:disc}

We now consider the limit $C \to +\infty$. If $C \gg
\max_{(i,\bx,\by)} |g_i(\bx,\by)|$ then $\ln\left(1+g_i/C\right)$
is approximately equal to $g_i/C$. From this remark we obtain the following result, which was shown by Cabrales and Sobel (1992), in the aggregate monotonic case.
\begin{proposition}\label{prop:Clarge}
Fix a game and functions $g_i$; fix a mixed strategy $\bq$
strictly dominated by a mixed strategy $\bp$. Assume either that the dynamics is aggregate monotonic (in the continuous-time sense) or that it is convex monotonic and that $\bq$ is pure. Then there
exists $\bar{C}$ in $\R$ such that for all $C\geq \bar{C}$, the
discrete dynamics (\ref{eq:discrete-gen}) eliminates strategy $\bq$.
\end{proposition}
\begin{proof}
There exists $\varepsilon>0$ such that
$$\forall (\bx,\by), \sum_{i \in I} (p_i-q_i) U_{i}(\by) \geq
\varepsilon$$
Therefore, under the above assumptions, there exists $\alpha>0$
such that
$$\forall (\bx,\by), \sum_{i \in I} (p_i-q_i) g_i(\bx,\by) \geq
\alpha$$
For $C$ large enough,
$$\forall (i,\bx,\by),  \left|\ln\left(1 +
\frac{g_i(\bx,\by)}{C}\right)-\frac{g_i(\bx,\by)}{C}\right| <
\frac{\alpha}{4C}$$
so that
$$\forall (\bx,\by), \sum_{i \in I} (p_i-q_i) \tilde{g}_i(\bx,\by)
\geq \alpha/2C>0$$
where $\tilde{g}_i=\ln C + \ln\left(1 + \frac{g_i}{C}\right)= \ln\left(C + g_i\right)$. The result
follows.
\end{proof}

Note that the constant $\bar{C}$ depends not only on the game and
the functions $g_i$, but also, through $\varepsilon$, on the
strategies $\bp$ and $\bq$. Compare proposition
\ref{prop:Clarge} and proposition
\ref{prop:discrete-conv-neg}. Note that for aggregate
monotonic dynamics, the order of the quantifiers (whether we first
fix the game, the functions $g_i$ and the strategies $\bp$ and
$\bq$, or we first fix the constant $C$) is crucial.

To conclude, note that if the constant $C$ depends on the step
$n$:
\begin{equation}
\label{eq:discrete-gen-Cn} x_i(n+1)=x_i(n) \frac{C_n+
g_i(\bx,\by)}{C_n + \sum_{k} x_k g_k(\bx,\by)}
\end{equation}
then we have:
\begin{proposition}
Assume that $C_n \to + \infty$ as $n\to + \infty$ and that
\begin{equation} \label{eq:discrete-alltheway}
\sum_{n \in \N} \frac{1}{C_n} = + \infty,
\end{equation}
Then for any game, any functions $g_i$ and any mixed strategies
$\bp$ and $\bq$ such that $\bp$ strictly dominates $\bq$, if the dynamics is aggregate monotonic or if it is convex monotonic and $\bq$ is pure,  then the discrete dynamics
(\ref{eq:discrete-gen-Cn}) eliminates strategy $\bq$.
\end{proposition}
\begin{proof}
This follows from proposition \ref{prop:Clarge}. The condition
(\ref{eq:discrete-alltheway}) is needed for the dynamics not
to slow down too much and ``stop".
\end{proof}
\end{appendix}


\begin{thebibliography}{99}
\begin{small}
\bibitem{Akin} Akin, E. (1980), ``Domination or Equilibrium", \emph{Math. Biosciences} 
 \textbf{50}, 239-250
 
\bibitem{Berger} Berger, U. and J. Hofbauer (2005), ``Irrational Behavior in the Brown-von
Neumann-Nash Dynamics", \emph{Games and Econ. Behav.}, forthcoming

\bibitem{Bernheim} Bernheim D.G. (1984), ``Rationalizable Strategic Behavior", \emph{Econometrica} \textbf{52}, 1007-1028

\bibitem{Brown-vonN} Brown, G.W. and J. von Neumann (1950), ``Solutions of Games by Differential Equations", in H.W. Kuhn and A.W. Tucker, Eds. 
\emph{Contributions to the Theory of Games I, 73-79}. Ann. of Math. Stud. \textbf{24}. Princeton University Press. 

\bibitem{Bjorn95}  J. Bj\"ornerstedt, ``Experimentation, Imitation and Evolutionary Dynamics, '' Department of Economics, Stockholm University, 1995. 

\bibitem{Bjornetal96} J. Bj\"ornerstedt, M. Dufwenberg, P. Norman, and J. Weibull, ``Evolutionary Selection 
Dynamics and Irrational Survivors, '' in ``Understanding Strategic Interaction: 
Essays in Honor of Reinhard Selten'' (W. Albers et al., Eds.), Springer-Verlag, Berlin, 1996, pp128-148
%
\bibitem{Cabrales2000} A. Cabrales, Stochastic replicator dynamics. Int. Econ. Rev. 41 (2000), pp. 451--81.
\bibitem{CabrSobel} Cabrales, A. and J. Sobel (1992), ``On the Limit Points of Discrete Selection Dynamics",
\emph{J. Econ. Theory} \textbf{57}, 407-420
\bibitem{Cress05} R. Cressman, Stability of the replicator equation with continuous strategy space, {\em Math. Soc. Sci.} \textbf{50} (2005), 127--147.
\bibitem{CressHof05} R. Cressman and J. Hofbauer, Measure dynamics on a one-
dimensional continuous trait space: theoretical foundations for adaptive dynamics, {\em Theoret. Popul. Biol.} \textbf{67} (2005), 47--59.
\bibitem{CrHoRi06} R. Cressman, J. Hofbauer, and F. Riedel, Stability of the replicator equation for a single species with a 
multi-dimensional continuous trait space, {\em J. Theoret. Biol.} \textbf{239} (2006), 273--288 
\bibitem{Dekel-Sco} Dekel, E. and S. Scotchmer (1992), ``On the Evolution of Optimizing Behavior",
\emph{J. Econ. Theory} \textbf{57}, 392-406%
\bibitem{FriedmanDyn} Friedman D. (1991), ``Evolutionary  Games in
Economics", \emph{Econometrica} \textbf{59}, 637-666
%
\bibitem{Gilboa-Matsui} Gilboa, I. and A. Matsui (1991), ``Social Stability and Equilibrium",
 \emph{Econometrica} \textbf{59}, 859-867

\bibitem{HeifShSp07a} A. Heifetz, C. Shannon C., and Y. Spiegel, The dynamic evolution of preferences, {\em Econ. Theory} \textbf{32} (2007a), 251--286 

\bibitem{HeifShSp07b} A. Heifetz, C. Shannon, and Y. Spiegel, What to maximize if you must, {\em J. Econ. Theory} \textbf{133} (2007b), 31--57 

\bibitem{HofSig} Hofbauer, J. and K. Sigmund (1998), \emph{Evolutionary Games and
Population Dynamics}, Cambridge University Press%

\bibitem{HofSan} J. Hofbauer and W. H. Sandholm, " Survival of dominated strategies under evolutionary dynamics". Theoretical Economics, 
forthcoming. 

\bibitem{HofWei} Hofbauer, J. and J.W. Weibull (1996), ``Evolutionary Selection against Dominated Strategies",
\emph{J. Econ. Theory} \textbf{71}, 558-573

\bibitem{Matsui} Matsui, A. (1992), ``Best-Response Dynamics and Socially Stable
Strategies", \emph{J. Econ. Theory} \textbf{57},
343-362
\bibitem{MertMoust} P. Mertikopoulos and A. L. Moustakas: "The emergence of rational behavior in the presence of stochastic perturbations", Ann. Appl. Probab.
vol. 20 (4), July 2010.

\bibitem{Nachbar} Nachbar, J. (1990) `` `Evolutionary' Selection Dynamics in Games: Convergence and
Limit Properties", \emph{Int. J. Game Theory}
\textbf{19}, 59-89

\bibitem{Pearce} Pearce, D. G., "Rationalizable Strategic Behavior and the Problem of Perfection", Econometrica 52, 1029--50, 1984

\bibitem{SamZhang} Samuelson, L. and J. Zhang (1992), ``Evolutionary stability in asymmetric games",
\emph{J. Econ. Theory} \textbf{57} , 363-391

\bibitem{Sandholm} Sandholm, W.H. (2011), \emph{Population Games and Evolutionary Dynamics}, MIT Press

\bibitem{Dynamo} W. H. Sandholm, E. Dokumaci, and F. Franchetti (2011). Dynamo: Diagrams for Evolutionary Game Dynamics, version 1.1. 

\bibitem{Smith} Smith, M. J. (1984). The stability of a dynamic model of traffic assignment-- an application 
of a method of Lyapunov. Transportation Science, \textbf{18}, 245--252.

\bibitem{TayJon} Taylor, P.D., and L. Jonker (1978), ``Evolutionary Stable Strategies and Game Dynamics",
\emph{Math. Biosci.} \textbf{40}, 145-156

\bibitem{Vio07} Viossat, Y., The Replicator Dynamics Does not Lead to Correlated Equilibria, Games and Econ. Behav. \textbf{59}, 397-407  (2007)

\bibitem{Vio08} Viossat, Y., Evolutionary Dynamics May Eliminate All Strategies Used in Correlated Equilibria,
 Math. Soc. Sci. 56, 27-43 (2008)
 
 
 
 \bibitem{Wei91} F. Weissing, ``Evolutionary stability and dynamic stability in a class of evolutionary normal form games", in R. Selten (ed.) \emph{Game Equilibrium Models: I. Evolution and Game Dynamics}, Springer Verlag, Berlin/New-York, 1991
 


\end{small}
\end{thebibliography}
\end{document}